\documentclass{article}
\usepackage[T1]{fontenc}
\usepackage[utf8]{inputenc}
\usepackage[english]{babel}
\usepackage{url}
\usepackage{color,hyperref,url,graphicx}
\hypersetup{colorlinks,linkcolor=[rgb]{.5,0,0},
  citecolor=[rgb]{.0,.4,.2},urlcolor=[rgb]{.2,.1,.4}}
\usepackage{amsmath,amssymb,amsfonts,amsthm,mathrsfs,stmaryrd}
\usepackage[all]{xy}

\usepackage{unicode}
\newtheorem{thm}{Theorem}[subsection]
\newtheorem{prop}[thm]{Proposition}
\newtheorem{lem}[thm]{Lemma}
\newtheorem{cor}[thm]{Corollary}
\theoremstyle{definition}
\newtheorem{df}[thm]{Definition}

\let\ro\mathcal 

\makeatletter
\let\c@equation\c@thm
\def\theequation{\thesubsection.\arabic{equation}}

 \let\leq\leqslant
 \let\geq\geqslant

\def\makecompact#1{\g@addto@macro#1{%
  \setlength{\itemsep}{\z@}\setlength{\parsep}{\z@}%
  \setlength{\topsep}{\z@}\setlength{\partopsep}{\z@}%
}}

\def\theequation{\@arabic\c@section.\@arabic\c@subsection.\@arabic\c@thm}
\def\endequation{\eqno \hbox{\@eqnnum}$$\@ignoretrue}
\makeatother

\DeclareMathOperator\Fil{Fil}
\DeclareMathOperator\Gr{Gr}
\DeclareMathOperator\Hom{Hom}
\DeclareMathOperator\Ext{Ext}
\DeclareMathOperator\Ker{Ker}

\DeclareMathOperator\id{id}
\DeclareMathOperator\haut{ht}

\def\Qp{{ℚ_p}}
\def\Qph{{ℚ_{p^h}}}
\def\Qpnr{ℚ_p^{\mathrm{nr}}}
\def\cris{_{\mathrm{cris}}}
\def\dR{_{\mathrm{dR}}}

\def\Cp{{ℂ_p}}
\def\tN{t_{\mathrm{N}}}
\def\tH{t_{\mathrm{H}}}
\def\trans#1{{}^{\mathrm{t}}#1}
\def\Bt{\widetilde{B}}
\def\Be{B_{\mathrm{e}}}
\def\limi{\varinjlim}

\def\BSE{\mathbf{BS^+}}
\def\BS{\mathbf{BS}}
\def\BC{\mathbf{BC}}
\def\BCO{\mathbf{BCO}}
\def\Czu{\mathbf{C}^{\acco{-1,0}}}

\def\defdelim#1#2#3{\def#1##1{{\left#2##1\right#3}}}
\defdelim\pa()
\defdelim\cro[]
\defdelim\acco\{\}
\defdelim\chev<>
\defdelim\abs\vert\vert
\defdelim\norm\Vert\Vert
\defdelim\bcro\llbracket\rrbracket

\def\dual#1{#1^{\smash{\scalebox{.7}[1.4]{%
  \rotatebox{90}{\textnormal{\guilsinglleft}}}}}}
\def\dualdR#1{#1\dR^{⋆}}
\def\application#1#2#3#4{\begin{array}{rcl}%
  \displaystyle#1&\longrightarrow&\displaystyle #2\\%
  \displaystyle#3&\longmapsto&\displaystyle #4\\\end{array}}

\begin{document}
\makecompact\itemize
\makecompact\enumerate

\title{Slope filtration on Banach-Colmez spaces}
\author{Jérôme Plût}
\date{March 1, 2012}
\maketitle

\begin{abstract}
We give a new proof of the ``weakly admissible implies admissible''
theorem of Colmez and Fontaine describing the semi-stable $p$-adic
representations. We study Banach-Colmez spaces, i.e. $p$-adic Banach
spaces with the extra data of a $\Cp$-algebra of analytic functions. The
``weakly admissible'' theorem is then a result of the existence of
\emph{dimension} and \emph{height} functions on these objects.
Furthermore, we show that the subcategory of Banach-Colmez spaces
corresponding to crystalline representations is naturally filtered by the
positive rationals.
\end{abstract}

\section*{Introduction}

The ``weakly admissible implies admissible'' theorem of Colmez and
Fontaine~\cite[Théorème~A]{CF2000} gives an equivalence of categories between
the semi-stable $p$-adic representations and an explicitly described
category of filtered~$(φ,N)$-modules. There exist at least five proofs of
the ``weakly admissible theorem'':
\begin{itemize}
\item the original proof by Colmez and Fontaine~\cite{CF2000} uses the
``almost surjectivity'' of certain formal series with coefficiens
in~$\Cp$;
\item Colmez' proof~\cite{Colmez2002EBDF}, which is a sheafification of the
proof of~\cite{CF2000};
\item Berger~\cite{Berger2008ED} related the filtered $φ$-module
and the $(φ,Γ)$-module attached to a crystalline representation
and used Kedlaya's filtration~\cite{Kedlaya2004monodromy} by
the slopes of the Frobenius;
\item Kisin~\cite{Kisin2006Crystalline} constructed a functor from filtered
$(φ,Γ)$-modules to a category of differential modules with a connexion;
\item most recently, Genestier and V. Lafforgue~\cite{GL2012HP} announced a
generalization of the proof by Kisin.
\end{itemize}
Of the above, our work is most closely related to the proof by
Colmez~\cite{Colmez2002EBDF}: although our proofs are different from his, the
category of Banach-Colmez spaces we introduce here is actually equivalent
to the category of ``Espaces de Banach de dimension finie''.

Following previous work~\cite{banach}, we adopt a geometric point of view
and consider \emph{spectral varieties}, which are topological spaces with
the data of a $\Cp$-Banach algebra of analytic functions and an
homeomorphism to the spectrum of this algebra.
In~\cite{banach}, we gave an analytic interpretation
of the ``fundamental lemma''~\cite[2.1]{CF2000}: we defined
the category of \emph{spectral Banach spaces} as the ``spectral varieties
in $p$-adic Banach spaces''. Finite-dimensional
$\Cp$-{}~and~$\Qp$-vector spaces naturally embed into this category, and
the category of \emph{effective Banach-Colmez spaces} is the full
subcategory of spectral Banach spaces which are analytic extensions of
finite-dimensional $\Cp$-vector spaces by finite-dimensional $\Qp$-vector
spaces. The fundamental lemma~\cite[Prop.~2.4.2]{banach} essentially
states that the dimension of the $\Cp-$ and~$\Qp$-parts of an effective
Banach-Colmez space~$E$ are well-defined. We name them respectively the
\emph{dimension} and \emph{height} of~$E$.

We define here the larger category of \emph{Banach-Colmez spaces} as
quotients of Banach-Colmez spaces by finite-dimensional $\Qp$-vector
spaces. Using the fundamental lemma, we show that this category is
abelian (Theorem~\ref{thm:ab}) and that the dimension and height
functions naturally extend to the larger category
(Corollary~\ref{cor:dh}).

We then describe a functor from the category of $φ$-modules to
Banach-Colmez spaces. Let~$K$ be a finite extension of~$\Qp$ and~$K_0$ be
the unramified subfield of~$K$. For any $φ$-module~$D$ with coefficients
in~$K_0$, we define
\begin{equation}
E(D) = \Hom_{K_0} (D, B^+\cris).
\end{equation}
We show that $E(D)$~has a natural structure as a Banach-Colmez space
(Prop.~\ref{prop:ED-BC}), of
dimension equal to the Newton number of~$D$, and height equal to the rank
of~$D$. Moreover, for any filtration~$\Fil$ on~$D_K = D ⊗_{K_0} K$, we
define
\begin{equation}
M(D,\Fil) = \Hom_{K} (D_K, B^+\dR) / \Hom_{K,\Fil} (D_K, B^+\dR).
\end{equation}
then $M(D,\Fil)$~is a $B^+\dR$-module of length equal to the Hodge number
of~$D$ and thus has a natural Banach-Colmez structure with
dimension~$\tH(D)$. Moreover, the canonical map~$E(D) → M(D,\Fil)$~is
analytic, and the ``weakly admissible'' theorem is then obtained simply
by counting dimensions and heights.

We finally study the essential image of the functor from $φ$-modules to
Banach-Colmez spaces: it is the full subcategory of \emph{oblique
Banach-Colmez spaces}, whose objects are extension of a finite-length
$B^+\dR$-module by a finite-dimensional $\Qp$-vector space. We give a
direct construction of an analogue of Kedlaya's filtration by the slopes
of the Frobenius for certain modules over the Robba ring: namely,
we show that oblique Banach-Colmez spaces have a canonical
Harder-Narasimhan filtration~(Theorem~\ref{thm:fil}) and that the stable
objects of slope~$μ = d/h$ are the Banach-Colmez spaces corresponding to
the isocrystal~$K_0[φ]/(φ^h-p^d)$ (Prop.~\ref{prop:stable}).

\section{Banach-Colmez spaces}

\subsection{Spectral Banach spaces}

We remind from~\cite{banach} that \emph{effective spectral Banach spaces}
are $p$-adic Banach spaces with an analytic structure provided by
an algebra of analytic functions.

Let~$\BSE$ be the category of effective spectral Banach spaces, and
$\Czu\BSE$ be the category of complexes~$V → E$, where~$V$~is a
finite-dimensional $\Qp$-vector space, $E$~is an effective spectral
Banach space, and the map~$V → E$~is injective. Then the
family~$\mathbf{Qis}$ of 
morphisms~$f: (V → E) → (V' → E')$ of~$\Czu\BSE$ such that the sequence
\[ 0 → V → E ⊕ V' → E' → 0 \]
is exact is a left multiplicative system
in the sense of~\cite[II.2.1]{Verdier1996}.

\begin{df}
The category~$\BS$ of \emph{spectral Banach spaces} is the localization
of the category~$\Czu\BSE$ relatively to the multiplicative
system~$\mathbf{Qis}$.
\end{df}

The natural functor from~$\BSE$ to~$\BS$, sending the
effective spectral Banach space~$E$ to the map~$(0 → E)$,
is fully faithful.

For any spectral Banach space~$X = (V →^d E)$, let~$H^0(X)$ be the
$p$-adic Banach space~$E / d(V)$. This construction is functorial and
defines a faithful functor from the category~$\BS$ to that of $p$-adic
Banach spaces.
Moreover, let~$X = (V →^d E)$, $X' = (V →^d E')$ be two
spectral Banach spaces; then a continuous linear map~$f: H^0(X) →
H^0(X')$ is in the image of~$H^0$ if, and only if, the \emph{graph}
of~$f$, i.e. the fibre product
\( E ×_{H^{0}(X')} E' \), is an analytic sub-space of the effective
spectral Banach space~$E × E'$.

Thereafter, by abuse of language, we shall identify the category~$\BS$
with its essential image by the faithful functor~$H^0$, i.e. we see the
spectral Banach spaces as $p$-adic Banach spaces. We call \emph{analytic}
the continous linear maps between spectral Banach spaces that are
morphisms of the category~$\BS$.

%
%

\begin{df}
A spectral Banach space~$X$ is \emph{étale} if there exists a
representative~$V → E$ of~$X$ such that $E$~is étale; it is
\emph{connected} if there exists a representative such that $E$~is
connected.
\end{df}

In particular, these definitions are compatible with the similar
definitions on the full subcategory of effective spectral Banach spaces.
Moreover, if $X$~is étale, then for all its representatives~$V → E$,
$E$~is étale; moreover, it is an effective spectral Banach space.

From the connected-étale sequence for effective spectral Banach spaces,
we deduce the analogous result for spectral Banach spaces.

\begin{prop}
Any analytic morphism from an étale spectral Banach space to a connected
one is zero.

Let~$X$ be a spectral Banach space. Then $X$~has a largest
connected sub-space~$X^0$, a largest étale quotient~$π_0(X)$, and the
sequence
\[ 0 → X^0 → X → π_0(X) → 0 \]
is exact. Moreover, this sequence is (non-canonically) split.
\end{prop}

\subsection{The abelian category of Banach-Colmez spaces}

\begin{df}
The category~$\BC$ of \emph{Banach-Colmez spaces} is the full subcategory
of $\BS$ whose objects have one representative of the form~$V → E$, where
$V$~is an étale spectral Banach space, and $E$~is an effective
Banach-Colmez space.
\end{df}

If $X$~is a Banach-Colmez space, then for any representative~$V → E$
of~$X$, $E$~is an effective Banach-Colmez space.

Let~$X$~be a Banach-Colmez space, and~$V_{-} → E$~be a representative
of~$X$. Then there exists a presentation of~$E$, of the form~$0 → V_{+} →
E → L → 0$, where $V_{+}$~is a finite-dimensional $\Qp$-vector space,
and $L$~is a finite-dimensional $\Cp$-vector space. We then say that the
diagram
\begin{equation}
0 → V_{-} → (V_{+} → E → L) → X → 0
\end{equation}
is a \emph{presentation} of~$X$. Given such a presentation, one defines
the \emph{dimension} of the presentation as being~$\dim_{\Cp} X$, and the
\emph{height} of the presentation as~$\dim_{\Qp} V_{+} - \dim_{\Qp}
V_{-}$.

From the fundamental lemma for effective Banach-Colmez
spaces~\cite[2.4.2]{banach}, we
deduce the following:
%
%

\begin{prop}\label{prop:dh-ker}
Let~$f: X → X'$ be a surjective morphism of Banach-Colmez spaces,
where~$X$ has a presentation of dimension~$d$ and height~$h$, and
$X'$~has a presentation of dimension~$d'$ and height~$h'$.

Then $X'' = \Ker f$~is a Banach-Colmez space, and it has a presentation
of dimension~$d'' = d - d'$ and height~$h'' = h - h'$.
\end{prop}

\begin{proof}
The fundamental lemma corresponds to the case where~$X$~is effective of
dimension one and~$X' = \Cp$. From this we deduce the general case as
follows.

Let~$V_{-} → (V_{+} → E → L) → X$ and~$V'_{-} → (V'_{+} → E' → L') → X'$
be presentations of~$X$ and~$X'$. By replacing~$f: X → X'$ by the map~$E
×_{X'} E' → E'$, we may assume that~$X'$ is effective and thus~$V'_{-} =
0$. 

Let~$L' = ⨁ L'_i$ be a decomposition of the $\Cp$-vector space~$L'$ as a
direct sum of $\Cp$-vector subspaces.
Then~$X'_i = X' ×_{L'} L'_i$ is an effective Banach-Colmez space and has
the presentation~$0 → V' → X'_i → L'_i → 0$. Let~$X_i = X ×_{L'} L'_i =
X ×_{X'} X'_i$: then the result is true for the map~$f: X → X'$ if, and
only if, it is true for all maps~$f_i: X_i → X'_i$.
We may therefore assume that~$L' = \Cp$. Moreover, replacing~$f$ by the
composed map~$X →^f X' → L'$ preserves the result, and we may thus
further assume that~$X' = L' = \Cp$.

Let now~$L = ⨁ L_i$ be a decomposition of~$L$, and let~$E_i = E ×_{L}
L_i$ and~$V_{-,i} = V_{-} ×_{E} E_i$. Then~$E_i$ is an effective
Banach-Colmez space, having the presentation~$0 → V_{+} → E_i → L_i → 0$,
and~$V_{-,i} ↪ E_i$ defines a Banach-Colmez space~$X_i$, as well as a
morphism~$X_i → X$. The result
for~$X → X'$ is then equivalent to the result for all composed maps~$X_i
→ X'$. We may therefore assume that~$L = \Cp$.

Let~$π_0(E)$ and~$E^0$ be the étale and connected components of~$E$; then
$E$~is isomorphic to~$π_0(E) × E^0$, and $E^0$~is a connected, effective
Banach-Colmez space of dimension one. The composed map~$g: E^0 ↪ E → X
→^f X' = \Cp$ is then analytic; since $f(X)$~is not étale by hypothesis,
$g$~is non-zero. By the fundamental lemma for the connected effective
Banach-Colmez space~$E$, $g$~is surjective and $\Ker g$~is a $\Qp$-vector
space of dimension~$\haut E^0 = \dim_{\Qp} V_{+} - \dim_{\Qp} π_0(V)$.
Therefore, $f$~is surjective, and its kernel has dimension~$\dim_{\Qp}
π_0(E) - \dim_{\Qp} V_{-} + \dim_{\Qp} \Ker g = \dim_{\Qp} V_{+} -
\dim_{\Qp} V_{-}$.
\end{proof}

\begin{cor}\label{cor:dh}
Let~$X$ be a Banach-Colmez space.
The dimension and height of any two presentations of~$X$ coincide.
\end{cor}

\begin{proof}
Apply~\ref{prop:dh-ker} to the identity morphism of~$X$.
\end{proof}

We call these the \emph{dimension} and \emph{height} of~$X$, and we
write~$\dim X$ and~$\haut X$.

\begin{prop}\label{prop:quo}
Let~$X$ be a Banach-Colmez space and~$X'$ be a sub-Banach-Colmez space of~$X$.
There exists a quotient Banach-Colmez space~$X'' = X/X'$;
moreover, $\dim X = \dim X' + \dim X''$ and~$\haut X = \haut X' + \haut X''$.
\end{prop}
\begin{proof}
Let~$V_{-} → (V_{+} → E → L) → X$ and~$V'_{-} → (V'_{+} → E' → L') → X'$
be presentations of~$X$ and~$X'$. As in the proof of
Prop.~\ref{prop:dh-ker}, we may reduce to the case where~$X = \Cp$.

If $X'$~is étale, then the quotient~$\Cp/X'$ exists and the result is true.
If on the other hand $X'$~is connected and non-zero,
then let~$0 → V' → E' → X' → 0$ be a presentation of~$X'$, with $E$~connected;
by the fundamental lemma for~$E'$,
the composed morphism~$E' → X' → \Cp$ is surjective.
In this case, the quotient~$\Cp/X'$ is zero.
\end{proof}

\begin{thm}\label{thm:ab}
The category of Banach-Colmez spaces is abelian. Moreover, there exist
unique functions \emph{dimension} and \emph{height}, additive on short
exact sequences, and such that
\[ \dim \Cp = 1, \quad \haut \Cp = 0; \qquad \dim \Qp = 0,\quad \haut \Qp = 1. \]
\end{thm}

\section{Banach-Colmez spaces from isocrystals and the ``weakly
admissible'' theorem}
\subsection{From isocrystals to Banach-Colmez spaces}

We use the notations of~\cite{CF2000}. Let $K_0$~be an unramified closed
subfield of~$\Cp$ with residue field~$k$; we write
$σ: K_0 → K_0$ for the absolute Frobenius automorphism. A
\emph{$K_0$-isocrystal} is a finite-dimensional $K_0$-vector space~$D$
with an injective $σ$-linear endomorphism~$φ: D → D$. The \emph{rank}
of~$D$ is its dimension~$h$ as a $K_0$-vector space. The
$K_0$-line~$⋀^{h} D$ is again an isocrystal, and the valuation~$v_p
(⋀^{h} φ)$ does not depend on the basis of~$D$. This integer is called
the \emph{Newton number} of~$D$, and is written~$\tN(D)$.

Let $D$~be an isocrystal. We define a contravariant functor~$E$ by
\begin{equation}\label{eq:ED}
E(D) = \Hom_{K_0[φ]} (D, B\cris).
\end{equation}
The isocrystal~$D$ is \emph{effective} if there exists a lattice~$\ro D$
of~$D$ such that~$φ(\ro D) ⊂ \ro D$. In this case, all elements of~$E(D)$
have their image contained in~$B^+\cris$.

If the residual field~$k$ of~$K_0$ is algebraically closed, then the
category of $K_0$-isocrystals is semi-simple, and its simple objects are
the objects
\begin{equation}
D_{d,h} = K_0[φ]/(φ^h-p^d), \quad \text{$d ∈ ℤ$, $h ≥ 1$}.
\end{equation}
In the general case, the rationals~$μ = d/h$ appearing in the decomposition
of~$D ⊗_{K_0} \Qpnr$ are called the \emph{slopes} of~$D$.
For any left $K_0[φ]$-module $(M, φ)$ and for~$r ∈ ℤ$, we write~$M(r)$
for the twisted module~$(M, p^r φ)$. Finally, we define
\begin{equation}
E_{d,h} = E(D_{d,h}) = \acco { x ∈ B^+\cris, φ^h(x) = p^{d} x}
\end{equation}
and, for all~$d$, $E_{d} = E_{d,1}$.

\begin{lem}\label{lem:Eeta-surj}
Let $D$~be an isocrystal with slopes~$≥ 1$ and of rank~$h$. Then there
exists a $K_0[φ]$-linear map~$η: D → B^+\cris(1)^h$ such that the
composed map
\[ θ ∘ E(η): E_1^h → E(D) → \Hom_{K_0} (D, \Cp) \]
is surjective.
\end{lem}
\begin{proof}
Here we understand $E_1$ as the space~$\Hom_{K_0[φ]} (B^+\cris(1), B\cris)$;
thus $E(η): E_1^h → E(D)$ is actually the image of~$η: D → B^+\cris(1)^h$
under the contravariant functor~$E$.

If the lemma is true for any unramified extension~$K_0'$ of~$K_0$, then
it is true for~$K_0$; therefore we may assume that $K_0 = \Qpnr$.
Let~$d = \tN(D)$; then there exists~$e ∈ D$ such
that $(e, φ(e), …, φ^{h-1}(e))$~is a basis of~$D$ and~$φ^h(e) = p^d e$.
Since $D$~has slopes~$≥ 1$, we have $d ≥ h$, thus there exists~$a ∈
B^+\cris$ such that~$φ^h(a) = p^{d-h} a$ and~$θ(\det φ^i(a)_{i = 0,…,
h-1}) ≠ 0$. Let~$η: D → B^+\cris(1)^h$ be defined by $η(φ^i(e)) = p^i
φ^i(a)$. Then $η$~is $K_0[φ]$-linear and the condition~$θ(\det φ^i(a)) ≠
0$ implies that $θ ∘ E(η)$~is surjective.
\end{proof}

\begin{prop}\label{prop:ED-BC}
For any isocrystal~$D$, $E(D)$~has a canonical structure as a
Banach-Colmez space, of dimension~$d = \tN(D)$ and height~$h = \dim_{K_0}
D$. Moreover, for any $K_0[φ]$-linear map~$f: D → D' ⊗_{K_0} B^+\cris$,
the deduced map~$E(f): E(D') → E(D)$ is analytic.
\end{prop}

In particular, this defines a functor from the category of
$K_0$-isocrystals to that of Banach-Colmez spaces.

\begin{proof}
For any~$r ∈ ℤ$, let~$D(r)$ be the Tate twisted isocrystal~$(D, p^{r}
φ)$. Then the slopes of~$D(r)$ are the slopes of~$D$, translated by~$r$.
By using a Tate twist, we may assume that $D$~is effective.

We now prove the proposition by induction on the maximal slope of~$D$. If
all the slopes of~$D$ lie in the interval~$[0,1]$, then there exists a
lattice~$\ro D$ of~$D$ such that~$p \ro D ⊂ φ(\ro D) ⊂ \ro D$, and
therefore $D$~is the Dieudonné module of a $p$-divisible group. In this
case the result follows from~\cite[2.1.5]{banach}.

Assume that the proposition is true for all isocrystals with slopes~$≤ m$
for some integer~$m$, and let~$D$ be an isocrystal with slopes~$≤ m+1$.
By using the isocline decomposition for~$D$, we may assume that $D$~has
all its slopes in the interval~$[m,m+1]$. Let~$h = \dim_{K_0} D$.

%
Define an \emph{admissible pair} for~$D$ as a pair~$(Δ, α)$, where $Δ$~is
an isocrystal with slopes~$≤ m$ and $α: D → Δ ⊗_{K_0} B^+\cris$ is a
$K_0[φ]$-linear map such
that $E(α): E(Δ) → E(D)$, defined by~$E(α)(f) = (\id_{B^+\cris} ⊗ f) ∘
α$, is surjective and has a finite-dimensional
kernel (as a $\Qp$-vector space). Then any admissible pair for~$D$ gives
a Banach-Colmez structure on~$E(D)$. Moreover, let~$f: D → D'$ be a morphism
of isocrystals, $(Δ, α)$~an admissible pair for~$D$, $(Δ', α')$ an
admissible pair for~$D'$, and $g: Δ → B^+\cris ⊗_{K_0} Δ'$~a
$K_0[φ]$-linear map such that~$g ∘ α = α' ∘ f$. Then the graph of~$E(f)$
is the fibre product of the maps~$E(f) ∘ E(α') = E(α) ∘ E(g)$
and~$E(α)$. By the induction hypothesis for~$Δ$ and~$Δ'$, it is
analytically closed in~$E(Δ) × E(Δ')$, and therefore $E(f)$~is an
analytic morphism. Finally, applying this to the identity morphism of~$D$
proves that the analytic structure on~$E(D)$ does not depend on the
choice of the admissible pair~$(Δ, α)$.

\medskip

It remains to prove that the isocrystal~$D$ has
an admissible pair~$(Δ, α)$.
Since $D$~has slopes~$≤ m+1$, the isocrystal~$D(-1)$ has slopes~$≤ m$
and, by the induction hypothesis, $E(D(-1))$~is a Banach-Colmez space.
Moreover, multiplication by a generator~$t$ of~$ℤ_p(1)$ in~$B^+\cris$
gives the left exact sequence
\begin{equation}
0 → E(D(-1)) →^{× t} E(D) →^{π} \Hom_{K_0} (D, \Cp).
\end{equation}
Let~$Δ = D(-1) ⊕ K_0(1)^h$
and~$α: D → Δ ⊗_{K_0} B^+\cris$ be the map defined by
$α(x) = (x ⊗ t) ⊕ η(x)$,
where $η: D → B^+\cris(1)^h$ is defined as in Lemma~\ref{lem:Eeta-surj}.
Then $α$~is $K_0[φ]$-linear and the image of~$E(α)$ contains
both~$E(D(-1))$ and representatives of~$\Hom_{K_0} (D, \Cp)$;
therefore, $E(α)$~is surjective.
The kernel of~$E(α)$ is isomorphic to~$\Hom_{K_0[φ]} (D, K_0(1)^h)$,
which is a $\Qp$-vector space of dimension~$h$.
Therefore, $(Δ, α)$~is an admissible pair for~$D$.

Since~$\tN(D(-1)) = d-h$ and~$\dim_{K_0} D(-1) = h$,
$E(D(-1))$~has dimension~$d-h$ and height~$h$ by the induction hypothesis;
therefore, $E(Δ)$~has dimension~$d$ and height~$2h$,
and $E(D)$~is a Banach-Colmez space of dimension~$d$ and height~$h$.
\end{proof}

%

\begin{prop}\label{prop:hom-ED}
Let $D$, $D'$ be two isocrystals. Then
\[ \Hom_{\BC} (E(D'), E(D)) = \Hom_{K_0[φ]} (D, D' ⊗_{K_0} B^+\cris). \]
\end{prop}
\begin{proof}
We know that any $K_0[φ]$-linear map~$D → D' ⊗_{K_0} B^+\cris$ is
analytic by Proposition~\ref{prop:ED-BC}. Conversely, any analytic
map~$f: E(D) → E(D')$ defines a $B^+\dR$-linear map~$f^+\dR:
E(D)^+\dR → E(D')^+\dR$. Since we also have~$f^+\dR(E(D)) ⊂ E(D'))$,
we see that $f$~comes from a $B^+\cris[φ]$-linear map~$D
⊗_{K_0} B^+\cris → D' ⊗_{K_0} B^+\cris$.
\end{proof}

\subsection{$B^+\dR$-modules as Banach-Colmez spaces}

Let~$K$ be a closed subfield of~$\Cp$ with discrete valuation and~$K_0$
be its maximal non-ramified subfield. A \emph{filtered $K$-isocrystal} is
a pair~$(D, \Fil)$, where $D$~is a $K_0$-isocrystal, and $\Fil$~is an
exhaustive and separated decreasing filtration on the $K$-vector
space~$D_K = D ⊗_{K_0} K$. By abuse, we shall sometimes write~$D$ instead
of~$(D, \Fil)$ when the filtration is clear. The \emph{Hodge number}
of~$(D, \Fil)$ is the integer~$\tH(D,\Fil)$ defined by
\begin{equation}
\tH(D,\Fil) = ∑_{i ∈ ℤ} i · \dim_K (\Fil^i/\Fil^{i+1} D_K).
\end{equation}
Since $D$~is finite-dimensional, this sum is finite. We also define 
\begin{gather}
V(D,\Fil) = \Hom_{φ,\Fil}(D, B^+\cris),\\
\notag
M(D,\Fil) = \Hom_{K} (D, B^+\dR) / \Hom_{K,\Fil} (D, B^+\dR).
\end{gather}
Together with~$E(D)$ as defined in~(\ref{eq:ED}), they make a left-exact
sequence of $\Qp$-vector spaces:
\begin{equation}\label{eq:VEM}
0 → V(D,\Fil) → E(D) → M(D,\Fil).
\end{equation}
The spaces~$E(D)$ and~$M(D,\Fil)$ are contravariant analogues of the
spaces~$V^0\cris$ and~$V^1\cris$ of~\cite[§5]{CF2000}.


\begin{prop}\label{prop:BdR-BC}
There exists a fully faithful functor from the category of finite-length
$B^+\dR$-modules to that of Banach-Colmez spaces,
extending the inclusion of finite-dimensional vector spaces over $ℂ_p$.
Moreover, for any filtered isocrystal~$D$,
this analytic structure on~$M(D,\Fil)$ makes the sequence~(\ref{eq:VEM})
a left exact sequence of Banach-Colmez spaces.
\end{prop}

We prove this proposition in Lemmas~\ref{lem:M-BC} to~\ref{lem:M-plein}.

\begin{lem}\label{lem:M-BC}
Any finite-length $B^+\dR$-module~$M$
has a structure as a Banach-Colmez space.
\end{lem}

\begin{proof}
It is enough to prove this for
the quotient spaces~$B_m = B^+\dR / \Fil^m B^+\dR$.
This module inserts in the exact sequence~\cite[5.3.7(ii)]{Fontaine1994Corps}
\begin{equation}\label{eq:EmBm}
0 → \Qp(m) → E_m → B_m → 0,
\end{equation}
where~$E_m = \acco{ x ∈ B^+\cris, φ(x) = p^m x}$.
Let~$D_m$ be the isocrystal~$K_0[φ]/(φ-p^m)$;
then~$E(D_m) = E_m$, thus by Proposition~\ref{prop:ED-BC},
$E_{m})$~is a Banach-Colmez space.
Therefore, its quotient $B_m$ is a Banach-Colmez space.
\end{proof}

\begin{lem}\label{lem:E-M-analytic}
Let~$(D, \Fil)$ be a filtered isocrystal. Then the map~$E(D) → M(D)$ is
analytic.
\end{lem}

\begin{proof}
It is enough to prove this in the case where~$D = D_{d,h} =
K_0[φ]/(φ^h-p^d)$ and $\Fil$~is the filtration~$\Fil_m$, defined
by~$\Fil_m^{m-1} D_K = D_K$ and~$\Fil_m^{m} D_K = 0$.
In this case, $M(D, \Fil_m) = \Hom_{K} (D, B_m)$.

Let~$d = qr + h$ be
the Euclidean division of~$d$ by~$h$, and let~$η = (η_j)_{j=1,…,h}:
D_{r,h} → B^+\cris(1)$ as in Lemma~\ref{lem:Eeta-surj}; also let~$u ∈
E_1$ be such that~$θ(u) = 1$. According to Proposition~\ref{prop:ED-BC},
the analytic structure on~$E(D) = E_{d,h}$ is given by
\begin{equation}
\application{E_1^{qh} ⊕ E_{r,h}}{E_{d,h}}
{\pa{(x_{i,j})_{\begin{subarray}{l}i=0,…,q-1\\j=1,…,h\end{subarray}}; \quad y}}
{∑ u^i η_j x_{i,j} + t^q y }
\end{equation}
and by Lemma~\ref{lem:M-BC}, that on~$B_m$ is given by
\begin{equation}
\application{E_1^m}{B_m}{(x_i)_{i=0,…,m-1}}{∑ u^i x_i \pmod{\Fil^m}}
\end{equation}
The graph of the reduction map~$f: E(D) → M(D,\Fil)$ is the fibre product
of these two maps. It is analytically closed in~$E_1^{m+qh} × E_{r,h}$;
therefore, $f$~is analytic.
\end{proof}

\begin{lem}\label{lem:B+dR-f-an}
Let $f: M→ M'$ be a $B^+\dR$-linear map between finite-length
$B^+\dR$-modules. Then $f$~is analytic.
\end{lem}
\begin{proof}
It is enough to prove it for the multiplication map~$μ: B_m → B_m, x ↦
ax$, where~$a ∈ B_m$. Let~$\widehat a ∈ E_m$ be a lift of~$a$. Then the
graph of~$μ$ in~$E_m × E_m$ is the set
\begin{equation}
G = \acco {(x, y) ∈ E_m × E_m, \quad \widehat{a}x - y ∈ \Fil^m B^+\cris}.
\end{equation}
Let~$v ∈ E_m$ be such that~$v ≡ 1 \pmod{\Fil^m}$.
Then, for all~$(x,y) ∈ E_m × E_m$, we have~$z = \widehat{a} x - v y ∈ E_{2m}$,
and~$(x,y) ∈ G$ if, and only if, $z ∈ \Fil^m E_{2m}$.
By Lemma~\ref{lem:E-M-analytic},
$G$~is analytically closed, and therefore $μ$~is analytic.
\end{proof}

In particular, when applying this lemma to the identity map of any
$B^+\dR$-module~$M$, we see that the analytic structure on~$M$ does not
depend on the choice of $B^+\dR$-generators.

\begin{lem}\label{lem:M-plein}
Let~$M$, $M'$ be two finite-length $B^+\dR$-modules. Then any analytic
map~$f: M → M'$ is $B^+\dR$-linear.
\end{lem}

\begin{proof}
It is enough to prove that, for any $B^+\dR$-module~$M$ of finite length,
$\Hom_{\BC} (M, \Cp) = \Cp$, i.e. that any analytic morphism is zero
on~$\Fil^1 B_m = B_{m-1}(1)$. Moreover, given the exact sequence of
Banach-Colmez spaces
\begin{equation}
0 → B_{m-1}(1) → B_m → \Cp → 0,
\end{equation}
it is again enough to prove this in the case where~$m = 2$.
According to Lemma~\ref{lem:M-BC}, an analytic structure on~$B_2$
is given by the exact sequence
\begin{equation}
0 → V → E_{2,2} → B_2 → 0,
\end{equation}
where~$E_{2,2} = \acco {x ∈ B^+\cris, φ^2 x = p^2 x}$, and~$V = E_{2,2} ∩
\Fil^2 B^+\cris$ is a $\Qp$-vector space of dimension~two.
Then $\Hom_{\BC} (B_2, \Cp)$ is given by the left-exact sequence
\begin{equation}
0 → \Hom_{\BC} (B_2, \Cp) → \Hom_{\BC} (E_{2,2}, \Cp) →^{F}
\Hom_{\BC} (V, \Cp).
\end{equation}
The space~$E_{2,2}$ is isomorphic to~$E_1 ⊗_{ℚ_p} ℚ_{p^2}$.
Let~$\acco{e, φ(e)}$ be a $ℚ_p$-basis of~$ℚ_{p^2}$,
and define~$f_0, f_1 ∈ \Hom_{\BC} (E_{2,2}, \Cp)$ by
\begin{equation}
f_i (x_0 e + x_1 φ(e)) = θ(x_i), \quad \text{for~$i ∈ \acco{0,1}$};
\end{equation}
Assume that the restriction map
$F: \Hom_{\BC} (E_{2,2}, \Cp) → \Hom_{\BC} (V, \Cp)$ is identically zero.
Then~$F(f_0) = F(f_1) = 0$ imply that any element~$x ∈ V$ may be
written as~$x = x_0 e + x_1 φ(e)$ with~$θ(x_0) = θ(x_1) = 0$. Since~$x_i
∈ E_1$, this implies that~$x_i ∈ \Fil^1 E_1 = \Qp(1)$
and therefore~$V ⊂ ℚ_{p^2}(1)$, which is absurd.
Therefore, $F$~is not zero, and its
kernel~$\Hom_{\BC} (B_2, \Cp)$ thus has dimension (at most) one.
\end{proof}

\begin{lem}\label{lem:lf-ext1}
Let~$E$ be an effective, connected one-dimensional Banach-Colmez space.
Then:
\begin{enumerate}
\item any effective extension of~$\Cp$ by~$E$ is trivial;
\item the map~$θ ∘ - : \Hom (E, E_1) → \Hom (E, \Cp)$ is surjective.
\end{enumerate}
\end{lem}

\begin{proof}
(i) Let~$E'$ be an effective extension of~$\Cp$ by~$E$.
For any morphism~$u: E' → \Cp$, $u(E)$~is either étale or~$\Cp$.
Since $E$~is connected, there exists~$u: E' → \Cp$ such that~$u(E') = \Cp$.
Define $V' = \Ker u$.
The map~$E' ↪ E$ defines a morphism between the exact sequences
$0 → V' → E' →^{u} \Cp → 0$ and~$0 → V' → E → \Cp ⊕ \Cp → 0$.
Let~$ι$ be the injection of the first summand~$\Cp → \Cp ⊕ \Cp$.
By~\cite[Prop.~2.2.3]{banach}, there exists a $\Cp$-linear map~$λ: ℂ_p^2
→ V ⊗_{\Qp} \Cp(-1)$ such that~$E' = E(λ) = (V ⊗ E_1(-1)) ×_{V ⊗_{\Qp}}
ℂ_p^2$ and~$E = E(λ ∘ ι)$. Since $E$~is connected, the transpose map
$\trans{(λ ∘ ι)} : \Hom(V,\Qp) → \Hom(\Cp, \Cp)$ is injective;
therefore, for any $ℂ_p$-linear retraction~$ρ: ℂ_p^2 → \Cp$ of~$ι$,
there exists a $\Qp$-linear map~$μ: V → V$ such that the diagram
\begin{equation}
\xymatrix{
ℂ_p^2 \ar[r]^{λ\quad}\ar[d]^{ρ}
  & V ⊗_{\Qp} \Cp(-1) \ar[d]^{μ ⊗_{\Qp} \Cp(-1)}\\
\Cp\ar[r]^{λ ∘ ι\qquad} & V ⊗_{\Qp} \Cp(-1)\\
}\end{equation}
commutes. The pair~$(μ, ρ)$ defines a retraction of~$E' → E$.

(ii) By~\cite[Cor.~2.3.13]{banach}, $E$~is isomorphic to~$E_{1,h}$ for
some integer~$h ≥ 1$. Therefore, for any~$u: E → \Cp$, there exist~$u_0,
…, u_{h-1} ∈ \Cp$ such
that~$u = ∑ u_r θ φ^{r}$. Let~$f_0,…,f_{h-1} ∈ E_{h-1,h}$ be such
that~$θ(f_r) = u_r$ for all~$r$, and define~$f(x) = ∑ f_r φ^r(x)$.
Then~$f: E_{1,h} → E_1$ is analytic and~$θ ∘ f = u$.
\end{proof}

\begin{prop}\label{prop:ext-dR}
Let~$M$, $M'$ be two finite-length $B^+\dR$-modules and~$X$ be an
extension of~$M$ by~$M'$ as Banach-Colmez spaces. Then $X$~is a
finite-length $B^+\dR$-module.
\end{prop}

\begin{proof}
It is enough to prove this when~$M = M' = \Cp$; by commodity, we
write~$M' = \Cp(1)$. Let~$X$ be an analytic extension of~$\Cp$
by~$\Cp(1)$. The analytic map~$\Cp(1) → X$ corresponds to a pair~$(V' →
V, E' → E)$, where~$0 → V' → E' →^{π'} \Cp(1) → 0$~is exact and~$0 → V →
E → X → 0$~is a presentation of~$X$. Likewise, the analytic map~$X → \Cp$
corresponds to a map~$E_1 → \Cp$, where~$0 → V_1 → E_1 → X$~is another
presentation of~$X$. By taking fibre products, we may assume that~$E_1 =
E$; the exact sequence~$0 → \Cp(1) → X → \Cp → 0$ then pulls back to the
exact sequence of effective Banach-Colmez spaces
\begin{equation}
0 → E' → E → \Cp → 0.
\end{equation}
Moreover, since $\Cp(1)$~is connected, we may assume that~$E'$~is
connected. By Lemma~\ref{lem:lf-ext1}~(i), $E$~is therefore isomorphic
to~$E' ⊕ \Cp$.

The analytic structure on~$B_2$ is given by the exact sequence~$0 →
\Qp(1) → E_1 ⊕ \Cp(1) → B_2 → 0$. By~\ref{lem:lf-ext1}~(ii), there
exist~$u: E' → E_1$ such that~$θ ∘ u = π' : E' → \Cp(1)$. The map~$u ⊕
\id_{\Cp}: E' ⊕ \Cp → E_1 ⊕ \Cp(1)$ defines an analytic morphism~$f: X →
B_2$ extending the identity morphism of~$\Cp$. The extension~$X$ is then
given by the restriction of~$f$ to~$\Cp(1)$, which is an analytic
endomorphism of~$\Cp$ and therefore multiplication by an element
of~$\Cp$. We finally see that~$\Ext^1_{\BC} (\Cp, \Cp(1)) =
\Ext^1_{B^+\dR} (\Cp, \Cp(1)) = \Cp$.
\end{proof}

\subsection{Weakly admissible implies admissible}

A filtered isocrystal~$(D, \Fil)$ is called \emph{admissible} if
$V(D,\Fil)$~is a $\Qp$-vector space of dimension equal to the rank
of~$D$. It is called \emph{weakly
admissible} if~$\tH(D) = \tN(D)$ and, for any subobject~$D'$ of~$D$,
$\tH(D') ≤ \tN(D')$. The ``weakly admissible implies admissible'' theorem
was first proven in~\cite{CF2000}. We give an independent proof of this
theorem.

\begin{thm}\label{thm:FAA}
Let~$K$ be a closed subfield of~$ℂ_p$ with discrete valuation.
Any weakly admissible filtered $K$-isocrystal is admissible.
\end{thm}

\begin{proof}
For any filtered isocrystal~$D$, the left-exact sequence
\begin{equation}\label{eq:VEM2}
0 → V(D,\Fil) → E(D) → M(D,\Fil)
\end{equation}
is analytic by Lemma~\ref{lem:E-M-analytic}.
If $(D, \Fil)$~is weakly admissible, then $\tH(D) =
\tN(D)$ implies that~$\dim M(D) = \dim E(D)$;
and $\tH(D') ≤ \tN(D')$ for all subobjects~$D'$ of~$D$,
with $K$~having discrete valuation,
implies that $V(D)$~is a finite-dimensional $\Qp$-vector space
(an elementary proof is given in~\cite[4.5]{CF2000}).
Therefore, the cokernel~$X$ of~$E(D) → M(D)$ has dimension zero,
and is thus an étale Banach-Colmez space.
Since $M(D)$~is connected, this implies that~$X = 0$,
\emph{i.e.} that~$E(D) → M(D)$~is surjective.

Moreover, by counting heights on the exact
sequence~\ref{eq:VEM2}, we find that~$\dim_\Qp V(D,\Fil) = \haut V(D,\Fil) = 
\haut E(D) = \dim_{K_0} D$. Therefore, $(D, \Fil)$~is admissible.
\end{proof}
\section{The slope filtration}
\subsection{The universal extension of~$\Bt$ by~$\Qp$}

The goal of this part is to define a left adjunct to the functor
from finite-length~$B^+\dR$-modules to Banach-Colmez spaces.

Let~$\Bt = B\dR / B^+\dR$ and~$\Be = \acco {x ∈ B\cris, φ(x) = x}$.
Then~$\Bt = \limi B_m(-m)$ and~$\Be = \limi E_m(-m)$.
The direct limit of the exact sequence~\eqref{eq:EmBm} reads
\begin{equation}\label{eq:BeBt}
0 → \Qp → \Be → \Bt → 0
\end{equation}
where the objects are inductive limits of Banach-Colmez spaces.

For any Banach-Colmez space~$X$, define
\begin{gather}
\dualdR{X} = \Hom_{\mathrm{ind}-\BC} (X, \Bt) = \limi_{m} \Hom_{\BC} (X, B_m),\\
\notag
X^+\dR = \dual{(\dualdR{X})} = \Hom_{B^+\dR} (\dualdR{X}, \Bt), \quad
X\dR = X^+\dR ⊗_{B^+\dR} B\dR.
\end{gather}
In particular, for any finite-length $B^+\dR$-module $M$, we
have~$\dualdR{M} = \dual{M}$ and~$M^+\dR = M$; for any finite-dimensional
$\Qp$-vector space, $\dualdR{V} = \dual{V} ⊗_{\Qp} \Bt$ and~$V^{+}\dR
= V ⊗_{\Qp} B^+\dR$.

\begin{prop}\label{prop:ext1}
Let $M$~be a finite-length $B^+\dR$-module and $V$~be a
finite-dimensional $\Qp$-vector space. Then there is a canonical
functorial isomorphism
\[ \Ext^1_{\BC} (M, V) \;=\; \Hom_{B^+\dR} (M, V ⊗_{\Qp} \Bt). \]
\end{prop}

\begin{proof}
Let~$λ: M → V ⊗_{\Qp} \Bt$ be a $B^+\dR$-linear map. Then we may form the
diagram
\begin{equation}\label{diag:ext-univ}
\xymatrix{0 \ar[r] & V \ar[r] & V ⊗_{\Qp} \Be \ar[r] &
V ⊗_{\Qp} \Bt \ar[r] & 0 \\
0 \ar[r] & V\ar@{=}[u]\ar[r] & E(λ) \ar[u]\ar[r] & M\ar[u]^{λ} \ar[r] & 0
}\end{equation}
Since $M$~has finite length,
$λ$~factors through $V ⊗_{\Qp} B_m(-m)$ for some integer~$m$
and the diagram~(\ref{diag:ext-univ}) is actually
a diagram of Banach-Colmez spaces.
Therefore the fibre product~$E(λ)$~is a Banach-Colmez space.

To show that every extension of~$M$ is of this type,
it is enough to prove the result for~$M = B_m$.
We proceed by induction on~$m$. The result then follows from
the exact sequence~$0 → B_m(1) → B_{m+1} → \Cp → 0$ and the case where~$M =
\Cp$, proven in~\cite[2.2.3]{banach}.
\end{proof}

\subsection{Constructible Banach-Colmez spaces}

\begin{prop}\label{prop:constructible}
Let~$X$ be a Banach-Colmez space. The following conditions are
equivalent:
\begin{enumerate}
\item $X$~is an extension of a finite-length $B^+\dR$-module by a
finite-dimensional $\Qp$-vector space;
\item the envelope map~$ X → X^+\dR$ is injective;
\item there exists an analytic embedding of~$X$ in a finite-type
$B^+\dR$-module.
\end{enumerate}
\end{prop}

\begin{proof}
(i) $⇒$ (ii): by Proposition~\ref{prop:ext1}, we may assume that~$X =
E(λ)$ for a $B^+\dR$-linear map~$λ: M → V ⊗_{\Qp} B^+\dR$. Let~$X'$ be
the $B^+\dR$-submodule of~$V ⊗_{\Qp} B\dR$ generated by~$X$. 
Any~$v ∈ \dualdR{V}$ extends to a $B^+\dR$-linear
map~$v: V ⊗_{\Qp} B^+\dR → \Bt$, and the composed map~$X → X' → V ⊗ B\dR
→^v \Bt$ is analytic. Therefore, the sequence~$0 → \dualdR{M} → \dualdR{X} →
\dualdR{V} → 0$ is (right-)exact, and its dual sequence~$0 → V^+\dR →
X^+\dR → M → 0$~is exact. Since both maps~$V → V^+\dR$ and~$M → M$ are
injective, we deduce that $X→X^+\dR$~is injective.

(ii) $⇒$ (iii): since $X$~is a quotient of extensions of finite-length
$B^+\dR$-modules and finite-dimensional $\Qp$-vector spaces, we see that
$X^+\dR$~is a finite-type $B^+\dR$-module.

(iii) $⇒$ (i): let~$X ⊂ N$ be an analytic embedding in a finite-type
$B^+\dR$-module~$N$. We prove that there exists a finite-length
quotient~$M$ of~$N$ such that the composed map~$X → M$ is surjective with
étale kernel. We may assume that~$N = B_m$ and proceed by induction
on~$m$. The case~$m = 0$ is trivial; assume that the case~$B_m$ is known,
and let~$X ⊂ B_{m+1}$ be an analytic embedding. Let~$X' = X ×_{B_{m+1}}
B_m(1)$ and~$X'' = X/X'$; there exists a natural injective map~$ι: X'' →
\Cp$. By the fundamental lemma, $X''$~is either~$\Cp$ or étale. If $X''$~is
étale, then there exists an analytic section of~$X → X''$, so that we may
write~$X = X' × X''$ and the result on~$X$ is immediate.

Assume that~$X'' = \Cp$. By the induction hypothesis, there exists a
finite-length quotient~$B_r(1)$ of~$N$ such that~$X' → B_{m-1}(1) →
B_r(1)$ is surjective with étale kernel. The same is then true of
composed map~$X → B_m → B_r$, which proves the result.
\end{proof}

We say that a Banach-Colmez space is \emph{constructible} if it satisfies
the equivalent conditions of Proposition~\ref{prop:constructible}; and
that it is \emph{oblique} if it is constructible and~$X → X\dR$ is
injective.

\begin{prop}\label{prop:Hom-dR}
Let~$X$, $X'$ be two constructible Banach-Colmez spaces. Then $\Hom_{\BC}
(X, X')$ identifies with the set of all $B^+\dR$-linear maps~$f: X^+\dR →
(X')^+\dR$ such that~$f(X) ⊂ X'$.
\end{prop}
\begin{proof}
By replacing~$f$ by its graph map~$X → X × X'$, we may assume that
$f$~is injective. By Proposition~\ref{prop:constructible}, $X'$~inserts
in an exact sequence~$0 → V' → X' → M' → 0$, where $M'$~is a
finite-length $B^+\dR$-module and $V'$~is étale. Let~$V = V' ×_{X'} X$;
then $V$~is étale and therefore the map~$V → X$~is analytic.
By hypothesis, the composed map~$X → X' → M'$ factorizes through~$X^+\dR$
and is therefore analytic, therefore the cokernel~$X → M$ of~$V → X$ is a
sub-module of~$M'$. Finally, we get a morphism of exact sequences
between~$0 → V → X → M → 0$ and~$0 → V' → X' → M' → 0$; by
Proposition~\ref{prop:ext1}, this comes from a commutative square
\begin{equation}
\xymatrix{M \ar[r] \ar[d] & V ⊗_{\Qp} B^+\dR \ar[d] \\
M' \ar[r] & V' ⊗_{\Qp} B^+\dR}
\end{equation}
which defines an analytic structure on the morphism~$f: X → X'$.
\end{proof}

\begin{prop}\label{prop:dR-ex-fid}
The functor that maps a constructible Banach-Colmez space~$X$ to the
$B^+\dR$-module $X^+\dR$ is exact and faithful.
\end{prop}
\begin{proof}
The faithfulness is a consequence of the fact that, for any
constructible~$X$, $X ⊂ X^+\dR$. To prove exactness, it is enough to
prove that 
\begin{equation}
\Ext^1_{\BC} (X, \Bt) = \limi \Ext^1_{\BC} (X, B_m(-m)) = 0.
\end{equation}
Since $X$~is constructible, it inserts in a presentation~$0 → V → X → M →
0$, where $V$~is étale and $M$~is a finite-length $B^+\dR$-module. Using
the $\Ext$~exact sequence in the category of ind-Banach-Colmez spaces, we
see that it is enough to prove the result when $X$~is étale or a
$B^+\dR$-module.

If $X$~is étale, then any surjective map to~$X$ has a section, and
therefore~$\Ext^1_{\BC} (X, Y) = 0$ for all~$Y$. If $X$~is a
finite-length $B^+\dR$-module, the result follows from
Proposition~\ref{prop:ext-dR} and the fact that $\Bt$~is an injective
$B^+\dR$-module.
\end{proof}

\begin{prop}\label{prop:fil0}
Let $X$~be a constructible Banach-Colmez space. There exists a unique
filtration
\[ X = F^0 X ⊃ F^{+} X ⊃ F^{∞} X, \]
such that $F^0 X / F^+ X$~is étale, $F^+ X / F^{∞} X$ is connected and
oblique, and $F^{∞} X$~is a finite-length $B^+\dR$-module.
\end{prop}
\begin{proof}
We define $F^+ X$~as the connected part of~$X$, and~$F^{∞} X$ as the
$B^+\dR$-torsion part of~$(F^+X)^+\dR$.
Let~$X'$ be the image of~$X$ in the $B^+\dR$-module~$X^+\dR$.
Then $F^{∞} X/(X'∩ F^∞ X)$ is both torsion and free,
and therefore~$F^∞ X ⊂ X$.
\end{proof}

The object~$F^∞ X$ is essentially an extra step,
corresponding to slope~$∞$,
of the filtration of Theorem~\ref{thm:fil}.

\subsection{Harder-Narasimhan categories}

The following is a formalization of the Harder-Narasimhan
filtration~\cite{HN1974} in more general categories.

A \emph{Harder-Narasimhan category} is an exact category~$\ro C$~(in the
sense of~\cite[1.0.2]{Laumon1983}, together with an exact functor~$η$ (called
\emph{fibre functor}) to an abelian category~$\ro C_{η}$, and
additive functions~$d: \ro C → ℕ$ and~$r: \ro C_{η} → ℕ$, such that
\begin{enumerate}
\item[(HN1)] For any object~$X$ of~$\ro C$, $r(η(E)) = 0$ if, and only if,
$E = 0$;
\item[(HN2)] If $f: E → E'$ is a morphism of~$\ro C$ such that~$η(f)$ is an
isomorphism, then $d(E') ≥ d(E)$, with equality if, and only if, $f$~is
an isomorphism;
\item[(HN3)] For any object~$E$ of~$\ro C$, $η$ maps injectively the
subobjects of~$E$ to those of~$η(E)$;
\item[(HN4)] For any morphism~$f: F → G$ of~$\ro C$ that is the
compositum of a strict monomorphism~$F → E$ followed by a strict
epimorphism~$E → G$, there exists a decomposition~$f = v ∘ u ∘ g$ such
that $u$~is strict epi, $η(g)$~is an isomorphism, and $v$~is strict mono.
\end{enumerate}

\begin{prop}\label{prop:fibre-HN}
Assume that the category~$\ro C$ satisfies axioms~(HN1) to~(HN3), and
moreover that it has finite fibre products of strict monomorphisms,
compatible with the fibre product in the abelian category~$\ro C_{η}$.
Then $\ro C$ satisfies axiom~(HN4).
\end{prop}

\begin{proof}
Let~$f: F → G$ be the composite of a strict epi~$F → E$ and a strict
mono~$E → G$; let~$E' = \Ker (E → G)$. By the hypothesis of the
proposition, the fibre product~$F' = F ×_{E} E'$ exists and~$F' → F$~is a
strict mono; therefore, $F'' = F / F'$ exists and $F'' → F$~is a strict
epi. Since $\ro C$~is exact and $F → E$~is a strict mono, the amalgamated
sum~$G' = F' +_{F} E$ exists, and the pushout~$F' → G'$~is a strict mono.

Finally, the arrow~$F → G$ is the composite~$F → F' → G' → G$, where $F →
F'$~is a strict epi, $G' → G$~is a strict mono, and~$η(F') → η(G')$ is
the isomorphism in~$C_{η}$ between the co\"image and the image of~$η(f)$.
\end{proof}

For the remainder of this part,
we assume that $\ro C$~is a Harder-Narasimhan category.
As most of the proofs from~\cite{HN1974} directly apply in~$\ro C$,
we shall only detail here the parts that differ.

For any object~$E$ of~$\ro C$, we define~$r(E) = r(η(E))$ and~$μ(E) =
d(E) / r(E)$.
We call~$d(E)$, $r(E)$ and~$μ(E)$
the \emph{degree}, \emph{rank} and \emph{slope} of~$E$.
We write~$E'≼ E$ for a strict mono~$E' → E$,
and~$E' ≺ E$ if moreover~$E' → E$~is not an isomorphism.
We say that $E$~is \emph{semi-stable} if~$E ≠ 0$ and~$μ(E') ≤ μ(E)$
for any~$E' ≼ E$;
that $E$~is \emph{stable} if~$E ≠ 0$ and~$μ(E') < μ(E)$ for any~$E' ≺ E$.

%

Let~$F ≼ E$ be a strict subobject of~$E$. We say that $F$~is
\emph{costable} in~$E$ if, for all~$F ≺ F' ≼ E$, $μ(F') > μ(F)$.

\begin{lem}[{\cite[Lemma~1.3.5]{HN1974}}]\label{lem:sscs}
Let~$E$ be an object of~$\ro C$ and~$F_1$, $F_2$ be two strict subobjects
of~$E$ such that $F_1$~is semi-stable and $F_2$~is costable in~$E$. If
$F_1$~is not a strict subobject of~$F_2$, then~$μ(F_1) < μ(F_2)$.
\end{lem}

\begin{proof}
By hypothesis, the composite morphism~$F_1 → E → E/F_2$ is not zero. By
the axiom~(HN4), it factorizes as~$v ∘ g ∘ u$, where $v: F_1 → F'_1$~is
strict epi, $u: F'_2 → E/F_2$~is strict mono, and $η(g): η(F'_1) →
η(F'_2)$~is an isomorphism. By~(HN2), this implies that~$d(F'_2) ≥
d(F'_1)$; since~$r(F'_2) = r(F'_1)$, we have~$μ(F'_2) ≥ μ(F'_1)$.
Since~$F_1$~is semi-stable, $μ(F_1) ≤ μ(F'_1)$; moreover, since
$F_2$~is costable in~$E$ and~$F'_2 ≼ E/F_2$, $μ(F'_2) < μ(F_2)$. We
therefore have $μ(F_1) ≤ μ(F'_1) ≤ μ(F'_2) < μ(F_2)$.
\end{proof}

\begin{prop}[{\cite[Prop.~1.3.4]{HN1974}}]\label{prop:existe-sscs}
Let~$E$ be an object of~$\ro C$. Then $E$~contains a unique semi-stable
and costable strict subobject~$F$.
\end{prop}

\begin{prop}[{\cite[Lemma~1.3.8]{HN1974}}]\label{prop:fil-unique}
Let~$E$ be an object of~$\ro C$ and~$0 = F_0 ≺ F_1 ≺ .. ≺ F_n = E$ be an
increasing filtration by strict subobjects of~$E$, such that each
quotient~$F_i/F_{i-1}$ is semi-stable. Then the following conditions are
equivalent:
\begin{enumerate}
\item For each~$i$, $F_{i+1}/F_i$~is costable in~$E/F_i$;
\item The sequence~$(μ(F_i))$ is strictly decreasing.
\end{enumerate}
\end{prop}

A filtration satisfying these conditions is called a
\emph{Harder-Narasimhan} filtration on~$E$.

\begin{prop}[{\cite[Prop.~1.3.9]{HN1974}}]\label{prop:fil-existe}
Any non-zero object~$E$ of~$\ro C$ has a unique Harder-Narasimhan
filtration.
\end{prop}

Given such a filtration~$0 = F_0 ≺ … ≺ F_n = E$ and~$α ∈ ℚ$, we
define~$\Fil^{α} E = F_r$ where~$r$ is the largest index such
that~$μ(F_r) ≥ α$. The \emph{slope filtration} of~$E$ is the decreasing
filtration~$(\Fil^{α} E)$; its graded quotients are zero except for a
finite number of values~$α$, for which they are semi-stable.

\subsection{Harder-Narasimhan structure on oblique Banach-Colmez
spaces}

\begin{lem}\label{lem:oblique}
Let~$Y$ be an oblique Banach-Colmez space and~$X ⊂ Y$ be a
sub-Banach-Colmez space. Then:
\begin{enumerate}
\item $X$~is oblique;
\item $\overline{X} = X^+\dR ∩ Y$~is a sub-oblique-space of~$Y$, and the
quotient~$Y/\overline{X}$ is oblique;
\item if $Y/X$~is oblique, then~$X = \overline{X}$.
\end{enumerate}
\end{lem}
\begin{proof}
(i) We have~$X ⊂ Y ⊂ Y^+\dR$, which is finite free; therefore, $X$~is
oblique.

(ii) Since $X$~and~$Y$ are oblique, the natural map~$Y/\overline{X} →
Y\dR / \overline{X}\dR$ is injective, and therefore $Y/\overline{X}$~is
oblique.

(iii) If $Y/X$~is oblique, then $Y/X ⊂ Y\dR / X\dR$; this arrow factors
through~$Y/\overline{X}$, which means that~$Y/X ⊂ Y/\overline{X}$.
Therefore, $X = \overline{X}$.
\end{proof}

\begin{prop}\label{prop:oblique-hn}
Let~$\BCO$ be the category of oblique Banach-Colmez spaces. The
category~$\BCO$, equipped with the functor~$η(X) = X\dR$, the rank function
equal to the dimension over~$B\dR$, and the height function equal to the
dimension of the Banach-Colmez space~$X$, is a Harder-Narasimhan
category.
\end{prop}
\begin{proof}
The category~$\BCO$ is the full subcategory of~$\BC$ of objects~$X$ such
that $X^+\dR$ is finite free over~$B^+\dR$. Therefore, it is an exact
category, and the functor~$η$ is exact. Moreover, the rank and degree
functions are additive on exact sequences.

Axiom~(HN1) is evident. Let~$f: X' → X$ be such that~$η(f): X'\dR → X\dR$
is an isomorphism; then we have~$X' ⊂ X ⊂ X^+\dR = (X')^+\dR$, and
therefore $X'$~is a sub-object of~$X$. Since~$\dim(X') = \dim(X) -
\dim(X/X')$, axiom~(HN2) is satisfied. Moreover, if~$X', X'' ⊂ X$
and~$(X')\dR = (X'')\dR$, then we have~$X = X'$, and therefore~(HN3) is
satisfied.

Finally, let~$X', X'' ⊂ Y$ be strict subobjects in~$\BCO$, and let~$X =
X' ×_{Y} X''$. Since~$X^+\dR = (X')^+\dR ×_{Y^+\dR} (X'')^+\dR$,
$X^+\dR$~is a finite free $B^+\dR$-module, and therefore $X$~is oblique.
Let~$\overline{X} = X\dR ∩ Y$; then we have~$\overline{X} = \overline{X'}
∩ \overline{X''}$. Since $X'$~and~$X''$ are strict, we
have~$\overline{X'} = X'$ and~$\overline{X''} = X''$, and
therefore~$\overline{X} = X$ by Lemma~\ref{lem:oblique}. Therefore,
$X$~is a strict subobject of~$Y$. By Proposition~\ref{prop:fibre-HN},
(HN4)~is satisfied.
\end{proof}

Let~$M$ be a $B^+\dR$-module of length~$d$, $V$~be a $h$-dimensional
$\Qp$-vector space, and~$λ: M → V ⊗_{\Qp} \Bt$ be a $B^+\dR$-linear map.
Then the Banach-Colmez space~$X(λ) = (B_e ⊗_{\Qp} V) ×_{\Bt ⊗_{\Qp} V} M$
is oblique if, and only if, $λ$~is injective.
We say that $λ$~is \emph{stable} if it is injective and $X(λ)$~is stable.

\begin{lem}\label{lem:ladder1}
Let~$λ: B_d → V ⊗_{\Qp} \Bt$ be a stable $B^+\dR$-linear map and~$h =
\dim_{\Qp} V$. Let~$ι: E_{d,h} → (B^+\dR)^h$ be the envelope map.

Then there exists an analytic, injective map~$u: \Hom (V, E_{d-1,h}) →
\Hom (V, E_{d,h})$ such that the square 
\begin{equation*}
\xymatrix{
\Hom (V, E_{d,h}) \ar[r]^{\trans{λ} ⊗ ι} & \Hom_{B^+\dR} (B_d, \Bt^h) \\
\Hom (V, E_{d-1, h}) \ar[u]^{u} \ar[r]^{\trans{λ} ⊗ ι} 
  & \Hom_{B^+\dR} (B_{d-1}, \Bt^h) \ar[u]}
\end{equation*}
is commutative.
\end{lem}
\begin{proof}
Let~$ι_\Cp = t ι: E_{1,h} → \Cp^h$ be the map defined by~$ι_\Cp(x) = θ(ι(x)) =
(θ φ^r(x))_{r=0,…,h-1}$. Then, by the fundamental lemma~\cite[Lemma
2.3.3]{banach}, the
map~$\trans{λ} ⊗ ı_\Cp: \Hom (V, E_{1,h}) → \Hom_{B^+\dR} (B_d, \Cp(-1))^h$
is surjective, and its kernel~$V'$ is a $h$-dimensional $\Qph$-vector space.
Let~$c$ be a non-zero element of~$V'$, and define~$u: \Hom (V, E_{d-1,h})
→ \Hom (V, E_{d,h})$ by~$u(f)(v) = f(v) c(v)$. Then $u$~is analytic and
injective, and the square diagram is commutative as required.
\end{proof}

\begin{lem}\label{lem:ladder2}
Let~$λ: B_d → V ⊗_{\Qp} \Bt$ be a stable $B^+\dR$-linear map and~$h =
\dim_{\Qp} V$. Let~$ι: E_{d,h} → (B^+\dR)^h$ be the envelope map.
Then $(\trans{λ} ⊗ ι) (\Hom (V, E_{d,h}))$ contains the image
of~$\Hom_{B^+\dR} (\Cp, \Bt^h)$ in~$\Hom_{B^+\dR} (B_d, \Bt^h)$.
\end{lem}
\begin{proof}
By applying $d-1$ times Lemma~\ref{lem:ladder1} in succession, we get a
commutative square
\begin{equation}
\xymatrix{
\Hom (V, E_{d,h}) \ar[r]^{\trans{λ} ⊗ ι} & \Hom_{B^+\dR} (B_d, \Bt^h) \\
\Hom (V, E_{1, h}) \ar[u]^{u} \ar[r]^{\trans{λ} ⊗ ι}
  & \Hom_{B^+\dR} (\Cp, \Bt^h) \ar[u].}
\end{equation}
By the fundamental lemma~\cite[Lemma~2.3.3]{banach}, the bottom map is surjective. Therefore,
the image of~$\trans{λ} ⊗ ι$ contains~$\Hom(\Cp, \Bt^h)$.
\end{proof}

\begin{prop}\label{prop:lf-stable}
Let~$λ: M → V ⊗_{\Qp} \Bt$ be a stable $B^+\dR$-linear map, $h =
\dim_{\Qp} V$ and~$d$ be the length of~$M$. Let~$ι: E_{d,h} → (B^+\dR)^h$
be the envelope map.

Then~$\trans{λ} ⊗ ı: \Hom (V, E_{d,h}) → \Hom (M, \Bt^h)$ is surjective.
\end{prop}
\begin{proof}
First assume that there exist non-zero~$M'$, $M''$ such that~$M = M' ⊕
M''$. Since $λ$~is injective, there exist $B^+\dR$-linear maps~$λ': M' →
V' ⊗ \Bt$ and~$λ'': M'' → V'' ⊗ \Bt$ such that $λ: M → V ⊗_{\Qp} \Bt$~is
the direct sum of~$λ'$ and~$λ''$. Then at least one of~$μ(X(λ'))$
and~$μ(X(λ''))$ is greater than~$μ(X(λ))$, which contradicts the
stability of~$λ$. Therefore, we may assume that $M$~is simple,
which implies that~$M = B_d$.
We now prove the proposition by induction on~$d$.
The case~$d = 1$ corresponds to the fundamental
lemma~\cite[Lemma~2.3.3]{banach}.

Let~$λ'$ be the composed map~$B_{d-1} →^{t} B_d → V ⊗ \Bt$.
Since $λ$~is stable, $λ'$~is also stable.
By the induction hypothesis,
$\trans{λ'} ⊗ ι: \Hom (V, E_{d-1,h}) → \Hom (B_{d-1}, \Bt^h)$~is surjective.
Since this map is analytic, its kernel~$V'$ is
a $h$-dimensional $\Qph$-vector space.
Multiplication in~$B^+\cris$ induces
an analytic map~$V' ⊗_{\Qp} E_{1,h} → \Hom (V, E_{d,h})$;
moreover, the image of this map is exactly
the kernel of~$\Hom (V, E_{d,h}) → \Hom (B_{d-1}(1), \Bt^h)$.
We get the following diagram:
\begin{equation}
\xymatrix{
0 \ar[r] & V' ⊗_{\Qp} E_{1,h} \ar[r]\ar[d]
& \Hom (V, E_{d,h})\ar[r]\ar[d]^{\trans{λ} ⊗ ι}
& \Hom (B_{d-1}(1), \Bt^h) \ar[r]\ar@{=}[d] & 0
\\
0 \ar[r] & \Hom (\Cp, \Bt^h) \ar[r] & \Hom (B_d, \Bt^h)\ar[r]
& \Hom (B_{d-1}(1), \Bt^h)\ar[r] & 0
}
\end{equation}
By counting dimensions and heights, we see that the first line is exact.
The second line is exact because $\Bt$~is an injective $B^+\dR$-module.

By Lemma~\ref{lem:ladder2}, the image of~$\trans{λ} ⊗ ι$ contains~$\Hom
(\Cp, \Bt^h)$; therefore, the left vertical map is surjective.
This proves that the map~$\trans{λ} ⊗ ι$~is surjective.
\end{proof}

\begin{prop}\label{prop:stable}
Let~$X$ be a stable oblique Banach-Colmez space. Then $X$~is isomorphic
to~$E_{d,h}$, where $d = \dim X$ and $h = \haut X$.
\end{prop}
\begin{proof}
The proof mirrors that of~\cite[Prop.~2.3.2]{banach}. By
Prop.~\ref{prop:ext1}, we know that there exists a $B^+\dR$-linear
map~$λ: M → V ⊗_{\Qp} \Bt$ such that $X$~is isomorphic to~$E(λ)$.
By Prop.~\ref{prop:lf-stable}, $\trans{λ} ⊗ ι: \Hom (V, E_{d,h}) → \Hom
(M, \Bt^h)$~is surjective. Let~$L_{d/h}$ be the division algebra
over~$\Qp$ with Brauer invariant~$d/h$; then $\trans{λ} ⊗ ι$~is
left $L_{d/h}$-linear. By counting dimensions and heights, its kernel has
dimension zero and height~$h^2$ and is therefore a line over~$L_{d/h}$.
Let~$a$ be a generator of~$\Ker \trans{λ} ⊗ ı$ over~$L_{d/h}$.

Define~$f: V ⊗ \Be → B\cris$ by~$f(v ⊗ b) = b\: a(v)$. We readily check
that~$(φ^h - p^d) ∘ f = 0$. To show that $f: E(λ) → E_{d,h}$~is an
isomorphism, it is enough to prove that~$f(E(λ)) ⊂ B^+\cris$ and $f: E(λ)
→ E_{d,h}$~is surjective.

Let~$x ∈ E(λ)$ and~$y = f(x) ∈ \Fil^{-d} B\cris$. Since~$(\trans{λ} ⊗
ι)(a) = 0$, we have for all~$r ∈ ℕ$: $φ^r(t^d y) ≡ 0 \pmod{\Fil^d
B^+\dR}$. Therefore, for all~$r$, $φ^r(y) = t^{-d} p^{-d} φ^r(t^d y) ∈
\Fil^0 B\cris$. By~\cite[5.3.7~(i)]{Fontaine1994Corps}, this implies that~$φ^2(y)
∈ B^+\cris$ and therefore~$y ∈ B^+\cris$.

By taking coordinates in~$V$, we may write~$a = (a_1,…,a_h)$ with~$a_i ∈
E_{d,h}$. Since $a$~generates $\Ker(\trans{λ} ⊗ ι)$ as a left
$L_{d/h}$-module, the determinant $\det(φ^r a_i)$~is an unit in~$ℤ_p(d)$.
Therefore, the map~$f: E(λ) → B^+\cris$ is injective. By counting
dimensions and heights, it follows that $f: E(λ) → E_{d,h}$~is surjective
and therefore an analytic isomorphism.
\end{proof}

\begin{prop}\label{prop:ext1-ED}
Let~$0 ≤ d ≤ d'$ and~$h ≥ 1$ be integers. Then $\Ext^1_{\BC} (E_{d',h},
E_{d,h}) = 0$.
\end{prop}
\begin{proof}
(\cite[Proposition 9.4.4]{These})
This result follows from Prop.~\ref{prop:hom-ED} and counting dimensions
and heights on the long $\Ext^1$ sequence in the abelian category~$\BC$.
\end{proof}

\begin{thm}\label{thm:fil}
Let $X$~be an oblique Banach-Colmez space. Let~$ℚ^+ = [0,∞[ ∩ ℚ$.
\begin{enumerate}
\item
$X$~has a unique decreasing filtration~$(\Fil^{α} X)_{α ∈ ℚ^+}$,
where $\Gr^{α} X = \Fil^{α} X / \limi_{β > α} \Fil^{β} X$
is semi-stable of slope~$α$.
\item For all~$α = d/h ∈ ℚ^+$, there exists a unique integer~$n_{α}(X)$
such that~$\Gr^{α} X$ is (non-canonically) isomorphic
to~$E_{d,h}^{n_{α}(X)}$.
\item For any analytic map~$f: X → X'$ and for all~$α ∈ ℚ^+$,
we have~$f(\Fil^{α} X) ⊂ \Fil^{α} X'$.
\item The filtration~$\Fil$ is (non-canonically) split.
\end{enumerate}
\end{thm}
\begin{proof}
(i) follows from Props.~\ref{prop:oblique-hn} and~\ref{prop:fil-existe}.

(ii) The stable case follows from Prop.~\ref{prop:stable}, and the
semi-stable case from~$\Ext^1_{\BC} (E_{d,h}, E_{d,h}) = 0$
(Prop.~\ref{prop:ext1-ED}).

(iii) follows from Prop.~\ref{prop:hom-ED}.

(iv) follows from Prop.~\ref{prop:ext1-ED}.
\end{proof}

\bibliographystyle{alpha}
\bibliography{biblio}

\noindent
Written at Université Versailles--Saint-Quentin.\\\noindent
Permanent e-mail address: \texttt{jerome.plut@normalesup.org}
\end{document}